\documentclass[10pt,a4paper]{amsart}
\usepackage{a4wide,amsmath, amsthm, amsfonts, amssymb, epsfig, mathabx}

\def\blanco{\square}
\def\negro{\sqbullet}

\usepackage[all]{xy}
		
\def\ead#1{\email{#1}}

\def\xto#1{\xrightarrow{#1}}
\def\^#1{\widehat{#1}}

\def\d{\delta}
\def\G{{\mathcal G}}
\def\B{{\mathcal B}}
\def\H{{\mathcal H}}
\def\A{{\mathcal A}}
\def\X{{\mathcal X}}
\def\C{{\mathcal C}}

	\theoremstyle{plain}
		\newtheorem{teo}{Theorem}
		\newtheorem{prop}[teo]{Proposition}
		
		\newtheorem{lema}[teo]{Lemma}
	\theoremstyle{definition}
		\newtheorem{defi}[teo]{Definition}
		
		\newtheorem{example}[teo]{Example}

		\def\Cacti{\mathcal{C}}\def\cact{\mathcal C}

\def\m3{ \raisebox{-8pt}{\includegraphics[scale=1.2]{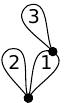}} - \raisebox{-8pt}{\includegraphics[scale=1.2]{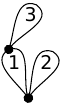}}}
        \def\Atwo{\A^{(2)}}

\author{Imma G\'alvez-Carrillo}
\address{Departament de Matem\`atica Aplicada III,
      Universitat Polit\`ecnica de Cata\-lunya\\ Escola d'Enginyeria de Terrassa, Carrer Colom 1, 08222 Terrassa (Bar\-celona), Spain}
\ead{m.immaculada.galvez@upc.edu}

\author{Leandro Lombardi}
\address{Departamento de Matem\'atica, Universidad de Buenos Aires\\
Ciudad Universitaria, Buenos Aires, Argentina}
\ead{llombard@dm.uba.ar}

\author{Andrew Tonks
}
\address{STORM, London Metropolitan University\\
166--220 Holloway Road, London N7 8DB, United Kingdom}
\ead{a.tonks@londonmet.ac.uk}

\title{An $\A_\infty$ operad in spineless cacti}

\begin{document}\maketitle

\begin{abstract}The d.g.\ operad $\C$ of cellular chains on the operad of spineless cacti~\cite{Kau07}
 is isomorphic to the Gerstenhaber--Voronov operad codifying the cup product and brace operations on the Hochschild cochains of an associative algebra, and to the suboperad $F_2\X$ of the surjection operad of~\cite{BF04}. Its homology is the Gerstenhaber operad $\G$.
We construct an operad map
$\psi:\A_\infty\longrightarrow \C$
such that $\psi(m_2)$ 
is commutative and  
$H_*(\psi)$ is the canonical map $\A\to \C\!om\to \G$. This formalises the idea that, since the cup product is commutative in homology,  its symmetrisation is a homotopy associative operation.
Our explicit $\A_\infty$ structure does not vanish on non-trivial shuffles in higher degrees, so does not give a map $\C om_\infty\to \C$. If such a map could be written down explicitly, it would immediately lead to a $\G_\infty$ structure on $\C$ and on Hochschild cochains, that is, a direct proof of the Deligne conjecture.
\end{abstract}

\section*{Introduction}
 
The Hochschild cohomology $H^*(A,A)$ of an associative algebra is an algebra over the Gerstenhaber operad $\G$, 
and Deligne's conjecture that this structure lifts to a suitable strong homotopy algebra structure on the cochain complex $C^*(A,A)$ has received much attention. Many proofs have appeared and there are now a large number of differential graded operads, all weakly equivalent to the singular chains on the little discs operad,  known to act on the Hochschild cochains. 
In particular, Tamarkin~\cite{Tam98b} showed that  $C^*(A,A)$ is a homotopy Gerstenhaber algebra in the operadic sense: the (quasi-free, minimal) resolution $\G_\infty$ of the Gerstenhaber operad acts on $C^*(A,A)$, via an operad $\B_\infty$ considered in~\cite[\S5.2]{GJ94}. Later, McClure--Smith~\cite{McS02} showed that the Gerstenhaber--Voronov operad $\H$ encoding the cup product and brace operations on cochains~\cite{GV} is also equivalent to the chains on the little discs. 
Extensions of the Deligne conjecture from associative to $\A_\infty$-algebras $A$ were given in~\cite{KS00,KauSch10}, and a cyclic version in~\cite{Kau08}. 

The aim of this note is modest in comparison: we provide a first step towards a possible {\em explicit} operad map from $\G_\infty$ to $\H$ which in homology is the identity map on $\G$.
 Concretely, we 
define an operad map 
\begin{equation}\A_\infty\longrightarrow F_2\X,\label{mainmap}\end{equation}
which in homology is just the canonical map $$\A\to \C om\to \G.$$          
Here $\A$ and $\C om$ denote the associative and commutative operads, $\A_\infty$ is the operad encoding associative algebras up to homotopy, and $F_2\X$ is an operad
isomorphic to $\H$ which is given as a certain suboperad of the surjection operad $\X$ of Berger--Fresse (see \cite{McS02,BF04} and \cite[\S5]{McS03}). Alternatively, we can identify $\H$ with the operad $\cact$ of cellular chains on Kaufmann's 
topological operad  of \emph{spineless cacti}, by  \cite[Proposition 4.9]{Kau07} (see also~\cite{Kau05}). Operads of cacti were first introduced by Voronov~\cite[\S2.7]{Vo05} to codify the Batalin--Vilkovisky structure discovered in string topology, compare~\cite[\S2]{CJ02}.

The lowest degree operation of $\cact$ codifies the cup product of Hochschild cochains. This product is not commutative but, since its failure to be so is the boundary of Gerstenhaber's brace operation, it becomes commutative in homology.  Thus we require the comparison map \eqref{mainmap} to send the binary product in $\A_\infty$ to the symmetrisation of the cup product. Of course, the symmetrisation of an associative operation need not be associative.

\section{The operads of surjections and cacti}

Let $\Bbbk$ be a commutative ring. A {\em (dg-)operad}\/ $\mathcal{O}$ is a sequence of differential graded $\Bbbk$-modules $\mathcal{O}(n)$, $n\geq 0$, together with composition operations
$$\circ_i\colon \mathcal{O}(m)\otimes \mathcal{O}(n)\longrightarrow \mathcal{O}(n+m-1),\qquad 1\leq i\leq m,\quad n\geq 0,$$
satisfying the operadic relations
\begin{enumerate}
\item $(a\circ_i b)\circ_j c=(-1)^{|b||c|}(a\circ_j c)\circ_{i+n-1}b$ if $1\leq j<i$ and $c\in\mathcal{O}(n)$.
\item $(a\circ_i b)\circ_j c=a\circ_i( b\circ_{j-i+1}c)$ if $b\in\mathcal{O}(m)$ and $i\leq j <m+i$.
\end{enumerate}
In addition there is a two sided unit $1\in \mathcal{O}(1)$ for the operations $\circ_i$. 
These structures often appear in the literature under the name  \emph{non-symmetric} operad.

Let us recall from \cite{BF04} the definition of the {\em surjection operad} $\X$.

\begin{defi} \label{defiX}
Let $\X(n)$ be the graded module whose degree $k$ component $\X(n)_k$ is spanned by the non-degenerate surjections $$u: \{1, \dots, n+k \} \to \{1, \dots, n \},$$ that is, the surjective functions $u$  such that $ u(i) \neq u(i+1)$ for all $i$.  We often write a surjection $u$ as the sequence of its values, 
$$u =(u(1), \dots, u(n + k)).$$
The operad structure maps are defined by
\begin{align}\nonumber 
\X(m)_l &\otimes \X(n)_k   \xto{\circ_t} \X(m+n-1)_{k+l}\qquad(1\leq t\leq m)\\
v&\circ_t u = \!\!\!\!\!\!\!\!\!\!\!\!
\sum_{
1=j_0\leq \dots\leq j_r=n+k
}
\!\!\!\!\!\!\!\!
\!\!\!\!
\pm
(\beta v_0, \alpha u_1, \beta v_1, 
\dots , 
\alpha  u_r, \beta v_r )\label{compsign}
\end{align}
Here $r=|v^{-1}\{t\}|\geq1$, the number of occurrences of the value $t$ in $v$, and $u_p$, $v_p$ 
 are subsequences of $u$, $v$ given by
$$
u_p = 
(u(i))_{j_{p-1}\leq i\leq j_p}
,\qquad
v = (v_0, t , v_1, t,  \dots , v_{r-1}, t, v_r ).
$$
The images of these subsequences are relabelled by composing with functions $\alpha,\beta$ given by
$$\alpha(s) = s + t -1,\qquad \beta (s) = \begin{cases} s & \text{if } s < t, \\ s + n -1 & \text{if } s>t.\end{cases}$$
To fix the signs, we consider the {\em relative degree} of $(u(a),u(a+1),\dots,u(b))$ in $u:\{1,\dots,n+k\}\to\{1,\dots,n\}$ to be 
$$
|\{a\leq i\leq b-1\::\:u(i)=u(i')\text{ for some }i<i'\leq n+k\}|
$$ 
Then the $\pm$ sign in \eqref{compsign} is defined by the Koszul sign rule from the permutation of $2r$ symbols
$$
v_1,  \dots , v_r, \, u_1, \dots, u_r 
\mapsto u_1, v_1,\dots, u_r, v_r  $$
in which the degrees
of $v_q$ for $q\neq r$  and for $q=r$ are the relative degrees of $(t,v_q,t)$  and of $(t,v_r)$ in $v$, respectively, and 
the degree of each  $u_p$ is its relative degree in $u$.

The differential $\d: \X(n)_k \to \X(n)_{k-1} $ is given by 
$$\d(u) = \sum_{i=1}^{n+k} (-1)^{r_i}\, \d_i (u)$$
where $\d_i$ skips the $i$th entry,
$$\d_i(u) =  \big( u(1), \dots, \^{u(i)},\dots , u(n+k) \big).$$
The term $\d_i(u)$ is omitted if $u(i)$ is the {\em only} occurrence in $u$ of some value.
Otherwise, $r_i$ is the relative degree of $(u(1),\dots,u(i))$ in $u$ if $u(i)$ is not the last occurrence of a value, or the relative degree of $(u(1),\dots,u(j+1))$ in $u$ if $u(i)$ is the last occurrence and $u(j)$ is the penultimate  occurrence.\end{defi}
For the proof that this structure indeed defines a differential graded operad we refer the reader to~\cite[Proposition 1.2.7]{BF04}. 
\begin{example}\label{goodsign}
Consider the sequences $v=(1,2,1)\in\X(2)_1$ and $u\in\X(n)_k$. Then the composite $(1,2,1)\circ_1 u\in \X(n+1)_{k+1}$ 
is given by
$$
(1,2,1)\circ_1 u\;=\;\sum_{j=1}^{n+k} (-1)^{|u|_j} \;\ring{u}_j
$$
where $|u|_j$ is the relative degree of $(u(1),\dots,u(j))$ in $u$ 
and $\ring u_j$ is obtained by  replacing the $j$th value in $u$ by $(u(j),n+1,u(j))$,
$$
\ring u_j\,=\,(u(1),\dots,u(j-1),u(j),n+1,u(j),u(j+1).\dots,u(n+k)).
$$
\end{example}
The operad $\X$ has a natural filtration by suboperads
$$ F_1\X \subset F_2\X \subset 
\dots \subset F_m \X \subset \dots \subset \X$$
The $m$th stage is spanned by the surjections $u$ in which all subsequences of the form $(u(r_p))_{1\leq p\leq q}=(i,j,i,j,i,\dots)$, $i\neq j$,
 have length $q\leq m+1$.
\\

The surjection operad $\X$ is, with different sign conventions, termed the sequence operad $\mathcal S$ in~\cite{McS03}, and the second stage $F_2\X$ is just the so-called Gerstenhaber--Voronov operad $\H$ encoding cup product and brace operations on the Hochschild cochains of an associative algebra. 
The second stage of the filtration is also isomorphic to the cellular chains on the topological operad of \emph{spineless cacti} \cite[Proposition 4.9]{Kau07}, see also \cite[Section 4]{Sa09} and \cite{KLP}. Here we take this as our definition:
\begin{defi} \label{defiCacti} 
The (dg) operad  $\cact$ of spineless cacti is the suboperad $F_2 \X$ of the surjection operad.  
\end{defi}

Thus a non-degenerate surjection $u : \{1,\dots, n+k  \} \to \{1,\dots, n \}$ is a \emph{cactus} if it has no subsequence of the form $( i,j,i,j )$ for $i\neq j$. 

Alternatively, a cactus may be represented as a configuration of labelled circles embedded in the plane, with pairwise intersections at at most one point, such that the union of all circles and their interiors forms a contractible region. A root is given on the boundary.

The circles are usually called the {\em lobes} of the cactus, and are labelled $1,\dots,n$. The intersection points together with the root divide the boundary of the cactus into segments, called the {\em arcs} of the cactus. The positive (anticlockwise) orientation of the plane specifies a labelling $1,\dots,n+k$ of the arcs, beginning at the root. The map from  arcs to the lobes that contain them gives a non-degenerate surjection $u$ as above. Examples are given in Figure~\ref{cactfig}. 

We see that the degree $|u|=k$ of a cactus $u$ is the number of distinct intersection points not equal to the root, and the maximum degree of a cactus with $n$ lobes is $n-1$, when the intersection points and the root are all distinct. The terms in the boundary $\d u$ are the cacti obtained by contracting one arc. The operadic composition $v\circ_t u$ in $\C$ may be seen (modulo signs) as substituting $u$ into the  lobe $t$ of $v$, and distributing the subcacti at the intersection points on this lobe in all possible ways across the arcs of $u$.

\begin{figure}
\center\begin{tabular}{ccccc}
 \raisebox{-5pt}{\includegraphics[scale=1]{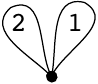}} & \raisebox{-5pt}{\includegraphics[scale=1]{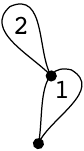}} &  \raisebox{-5pt}{\includegraphics[scale=1.4]{m3-b.pdf}}& \raisebox{-5pt}{\includegraphics[scale=1]{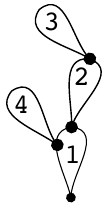}} &  \raisebox{-5pt}{\includegraphics[scale=0.55]{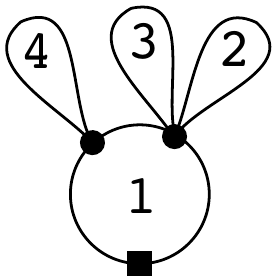}}  \\
$(1,2)$ & $(1,2,1)$ & $(2,1,3,1)$ & $(1,2,3,2,1,4,1)$ & $(1,2,3,1,4,1)$\\
\end{tabular}
\caption{Graphical representations of cacti.\label{cactfig}}
\end{figure}

\section{Commutativity and associativity}

The element $u=(1,2)=\!\!\text{\includegraphics{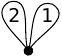}}\:$ of the cactus operad $\C$ encodes a binary product that is clearly associative:
$$(1,2) \circ_1 (1,2) = (1,2,3) = (1,2) \circ_2 (1,2).$$
In the homology of the operad it becomes commutative, since $(2,1)-(1,2)$ is the boundary of the element $(1,2,1)$,
$$\delta \Big( \,\raisebox{-10pt}{\includegraphics[scale=0.7]{ca121.pdf}} \Big) = \raisebox{-5pt}{\includegraphics[scale=0.8]{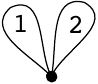}}  - \raisebox{-5pt}{\includegraphics[scale=0.8]{ca12.pdf}}.$$
On the other hand, we may consider the operation codified by 
$$ (1,2)+(2,1)=\raisebox{-5pt}{\includegraphics[scale=0.8]{ca12.pdf}}  + \raisebox{-5pt}{\includegraphics[scale=0.8]{ca21.pdf}}$$
This is commutative but no longer associative,
\begin{align*}
((1,2) + (2,1)) \circ_1 ((1,2) + (2,1)) \;\; &- \;\;((1,2) + (2,1)) \circ_2 ((1,2) + (2,1))\\ 
=\;\;(2,1,3) + (3,1,2)\qquad\qquad &\qquad\qquad -(1,3,2) - (2,3,1). 
\end{align*}
This failure of associativity will obviously vanish in homology, and we find, for example, that it is the boundary of the element $(1,3,1,2)-(2,1,3,1)$. 
Our goal is to show that this extends to an $\A_\infty$ structure:
\begin{teo}\label{mainthm}
There is an operad morphism $\psi:\A_\infty   \to    \Cacti$ given by
$$
\psi(m_2)=(2,1)+(1,2),\qquad\qquad \psi(m_n)=\sum_{u\in\C'_n}(-1)^{p(u)}\, u.
$$
Here $\C'_n$ is the set of all cacti $u:\{1,\dots,2n-2\}\to\{1,\dots,n\}\in\C(n)_{n-2}$  that contain no subsequences of the form $(i,j,i)$ with $j=i+1$ or $j<i$.
\end{teo}

In low degrees we have
\begin{equation}\label{m23}
m_2   \mapsto  \raisebox{-5pt}{\includegraphics[scale=0.9]{ca21.pdf}}  + \raisebox{-5pt}{\includegraphics[scale=0.9]{ca12.pdf}},\qquad\qquad 
m_3  \mapsto   \m3.
\end{equation}
as we saw above. 
The remainder of the paper is dedicated to the proof of the theorem, and in particular to making explicit the signs. For this we will need not only the operad $\A_\infty$ but also an operad $\Atwo_\infty$ codifying {\em homotopy matching dialgebras}~\cite{ZBG12}.
\\

An alternative description of $\C'_n$ is as follows. It is the set of cacti with no triple intersections between the $n$ lobes, such that if the lobe labelled $j$ is above the lobe labelled $i$ then $j-i\geq2$.
The condition on the lobe labels is just the forbidden subsequence condition, which in particular implies that lobes 1 and 2 must intersect at the root. The maximum possible degree of a cactus in $\C(n)$ satisfying this condition is $n-2$, which is obtained if and only if there are no triple intersections between lobes.

The value $n$ can occur only once in any sequence $u\in\C'_n$, as there are no lobes above the one labelled $n$.  For $n\geq 3$ we deduce that, since three lobes cannot intersect, $u$ must have the form
$$
(1,\dots,r,n,r,\dots,2)
\qquad\text{or}\qquad (2,\dots,r,n,r,\dots,1)
$$
for some $r\leq n-2$. 

Our proof of Theorem~\ref{mainthm} is based on the following inductive construction. 

\begin{defi}\label{ubw}
If $u$ is a cactus in $\C(n)_k$ such that $u^{-1}(\{n\})=\{i\}$, we let
$$
u^\blanco  \;=\; \sum_{j<i} (-1)^{k+|u|_j}\,\ring u_j,\qquad\qquad
u^\negro  \;=\;- \sum_{j>i} (-1)^{k+|u|_j}\,\ring u_j,
$$
Here 
$\ring u_j\in\C(n+1)_{k+1}$ is the sequence obtained from $u$ by replacing $u(j)$ with $(u(j),n+1,u(j))$, and $|u|_j$ is the relative degree of $(u(1),\dots,u(j))$ in $u$.
This definition is extended linearly to the submodule of $\C(n)$ spanned by those cacti $u$ with $u^{-1}(\{n\})$ a singleton.
\end{defi}

It is clear that the terms $\ring u_j$ of $u^\blanco$ and $u^\negro$  
are elements of $\C_{n+1}'$ if $u$ is an element of $\C_n'$, since they are obtained from adding a lobe labelled $n+1$ above any lobe except that which is labelled $n$. Conversely any  element  $v\in\C_{n+1}'$ appears as a term of $u^\blanco$ or $u^\negro$, where $u\in \C'_{n}$ is obtained by removing the lobe labelled $n+1$ from $v$.

\begin{lema}For  $n\geq 3$,
\begin{equation}\label{enu}
|\C'_n|=2(2n-5)!!=2\cdot(2n-5)\cdot(2n-7)\dots 5\cdot 3\cdot 1.
\end{equation}
\end{lema}
\begin{proof}
Observe that the set $\C'_2$ has just two elements, $(1,2)$ and $(2,1)$. Now for each element $(u(1),\dots,u(2n-2))\in\C'_n$ the construction in Definition~\ref{ubw} 
gives $2n-3$ terms in $u^\blanco$ and $u^\negro$. For distinct $u\in\C'_n$ these terms are distinct elements of $\C'_{n+1}$ 
and every element of  $\C'_{n+1}$ appears in this way. Therefore  $|\C'_{n+1}|=(2n-3)|\C'_{n}|$ and the result follows by induction.
\end{proof}

The operations $u^\blanco$ and $u^\negro$ behave nicely with respect to the differential and the operadic composition.

\begin{prop} \label{bncomp} 
For any cactus $u \in \C(n)_k$  such that $u^{-1}(\{n\})$ is a singleton,
\begin{eqnarray*}
u^\blanco-u^\negro&=&(-1)^k (1,2,1)\circ_1 u\;-\; u\circ_n(1,2,1),\\
\d (u^\blanco) - {(\d u)}^\blanco  & = & (-1)^k \;( (2,1) \circ_1 u -  u \circ_{n} (2,1) ) \\
\d (u^\negro)  - {(\d u)}^\negro   & = & (-1)^k \;( (1,2) \circ_1 u -  u \circ_{n} (1,2) )
\end{eqnarray*}
\end{prop}
\begin{proof}
Suppose  $u^{-1}(\{n\})=\{i\}$. 
For the first equation, observe from Example \ref{goodsign} and Definition \ref{ubw} 
that the difference between $(-1)^k(1,2,1)\circ_1u$  and
$u^\blanco-u^\negro$ is just the term 
$$
u \circ_n (1,2,1) \;= \; (-1)^{k - |u|_i} \, \ring u_i.
$$
Taking the differential of the first equation gives
\begin{align*}
\delta(u^\blanco) - \delta(u^\negro) \;=\;
& (-1)^{k} ( (2,1) -(1,2) ) \circ_1 u  \quad- (-1)^{k} u \circ_n ( (2,1) -(1,2) )  \\
 +&(-1)^{k-1}\;\; (1,2,1) \circ_1 \delta u \qquad\qquad\;- \delta u \circ_n (1,2,1) .
\end{align*}
Now we can use the first equation again for each term $v$ in $\d u$, since still $n$ appears only once in $v$. Hence
\begin{align*}
\delta(u^\blanco) - \delta(u^\negro) 
\;=\; & (-1)^{k} 
( 
(2,1) \circ_1 u
-(1,2)  \circ_1 u
- u \circ_n  (2,1) +  u \circ_n  (1,2) ) \\
 +& (\delta u)^\blanco -  (\delta u)^\negro.
\end{align*}
In any of the terms in this equation, each of the values $n$ and $n+1$ appear once only, and the equation splits into the two equations we require.
On the left hand side,
a term $v$ belongs to the expression $\delta(u^\blanco)$ if the value $n+1$ appears before the value $n$ in the sequence $v$, and it belongs to the expression $\delta(u^\negro)$ if the value $n+1$ appears after the value $n$.
On the right hand side, the value $n+1$ appears before the value $n$ in the terms of $(\delta u)^\blanco$, $(2,1) \circ_1 u$ or $u \circ_n (2,1)$, and the value $n+1$ appears after the value $n$ in the terms of $(\delta u)^\negro$, $(1,2) \circ_1 u$ or $u \circ_n (1,2)$.
\end{proof}

\begin{prop}\label{bncomp2}
For any cacti $u'\in\C(p)_k$, $u''\in\C(q)_{\ell}$,
 such that ${u'}^{-1}(\{p\})$, ${u''}^{-1}(\{q\})$ 
are singletons,
\begin{align*}
(u'\circ_i u'')^\blanco
=(-1)^{\ell }u'^\blanco\circ_i u''
,\qquad\qquad
(u'\circ_p u'')^\blanco
=(-1)^{\ell }u'^\blanco\circ_p u''+u'\circ_p u''^\blanco
\\
(u'\circ_i u'')^\negro
=(-1)^{\ell }u'^\negro\circ_i u''
,\qquad\qquad
(u'\circ_p u'')^\negro
=(-1)^{\ell }u'^\negro\circ_p u''+u'\circ_p u''^\negro
\end{align*}
for  $i\leq p-1$.
\end{prop}
\begin{proof}
By Proposition \ref{bncomp}, we have
\begin{align*}
(u'\circ_i u'')^\blanco- (u'\circ_i u'')^\negro
&=(-1)^{k+\ell }(1,2,1)\circ_1(u'\circ_i u'')
- (u'\circ_i u'')\circ_{p+q-1}(1,2,1)
\end{align*}
and by the operadic relations this can be rewritten as
\begin{align*}
(-1)^{k+\ell }((1,2,1)\circ_1u')\circ_i u''
-(-1)^{\ell } (u'\circ_p(1,2,1))\circ_iu''
&&(i<p)\\
(-1)^{k+\ell }((1,2,1)\circ_1u')\circ_p u''-u'\circ_p (u''\circ_{q}(1,2,1))&&(i=p)
\end{align*}
Now the operadic relation $(u'\circ_p(1,2,1))\circ_p u''=u'\circ_p ((1,2,1)\circ_1 u'')
$ implies
$$
(u'\circ_i u'')^\blanco
- (u'\circ_i u'')^\negro=\begin{cases}(-1)^{\ell }(u'^\blanco-u'^\negro)\circ_i u''&(i<p)\\(-1)^{\ell }(u'^\blanco-u'^\negro)\circ_p u''+u'\circ_p (u''^\blanco-u''^\negro)&(i=p)\end{cases}
$$
This equation splits into the black and white parts as required. On the left hand side a term belongs to $(u'\circ_i u'')^\blanco$
 if and only if $p+q$ appears before $p+q-1$. On
the right hand side  $p+q$ appears before $p+q-1$ 
in the terms of $u'^\blanco \circ_i u''$ 
because
$p+1$ appears before $p$  in $u'^\blanco$, and also in the terms of 
$u' \circ_p u''^\blanco$ 
because
$q+1$ appears before $q$  in $u''^\blanco$.
\end{proof}

\section{Homotopy matching dialgebras}

The operad of matching dialgebras $\Atwo$ was defined in \cite{ZBG12,Zinb10}, compare \cite[Exercise 9.7.4]{LV}. It is a generalisation of the associative operad $\A$ codifying two binary operations $\blanco$, $\negro$ with four associative laws. It is a binary, quadratic, Koszul and self-dual operad. The definition of the differential graded operad $\A^{(2)}_\infty$ 
of homotopy matching dialgebras can be expressed as follows. 

\begin{defi}
Let $\Atwo_\infty$ be the free operad on generators $m_\xi\in\Atwo(n)_{n-2}$ for each string $\xi=(\xi_1,\dots,\xi_{n-1})\in\{\blanco,\negro\}^{n-1}$, $n\geq2$. The differential is defined by
\begin{align}\label{partialxi}
\partial (m_\xi) \;\;=\;\; \sum_{i=1}^n (-1)^{(i-1)}\!\!\!\!
\sum_{\xi' \circ_i \xi'' = \xi}  (-1)^{q (p-i)} m_{\xi'} \circ_i m_{\xi''}
\end{align}
where $p+q-1=n$, $\xi'\in\{\blanco,\negro\}^{p-1}$, $\xi'\in\{\blanco,\negro\}^{q-1}$ and we write 
$$\xi' \circ_i \xi''\;\;=\;\;(\xi'_1, \dots, \xi'_{i-1}, \xi''_1, \dots, \xi''_{q-1}, \xi'_i, \dots, \xi'_{p-1} )\;\;\in\;\;\{\blanco,\negro\}^{n-1}.$$
\end{defi} 

The definition of the classical homotopy associative operad $\A_\infty$ may be given for comparison as the free operad with generators $m_n\in\A(n)_{n-2}$, $n\geq2$ and differential
\begin{align}\label{partialn}
\partial (m_n) \;\;=\;\; \sum_{i=1}^n (-1)^{(i-1)}\!\!\!\!
\sum_{p+q-1=n}  (-1)^{q (p-i)} m_p \circ_i m_q.
\end{align}
The following result is clear.
\begin{lema}\label{phirmk}
There is a morphism of  differential graded operads  
$$\phi:\A_\infty \longrightarrow \Atwo_\infty$$ 
given on generators by
\begin{align*}
\phi(m_n) \;\;= \sum_{\xi\in\{\blanco,\negro\}^{n-1}}\!\!\!
 m_\xi .
\end{align*}
\end{lema}

\noindent Theorem~\ref{mainthm} therefore follows from the following result.

\begin{teo}
There is a morphism of  differential graded operads  
$$\mu:\Atwo_\infty \longrightarrow \C$$
defined inductively on the generators by
\begin{align*}
&\;\mu(m_\blanco)=(2,1),&&
\;\mu(m_\negro)=(1,2),\\
&\mu(m_{\xi\blanco}) = \bigl(\mu(m_{\xi})\bigr)^\blanco, &&
\mu(m_{\xi\negro}) = \bigl(\mu(m_{\xi})\bigr)^\negro. 
\end{align*}
\end{teo} 
\begin{proof}
Since $\A^{(2)}_\infty$ is free as a graded operad, we need only show $\mu$ commutes with the differentials,  which will follow inductively from the relations
\begin{eqnarray*}
\d \mu (m_{\xi\blanco}) - ( \d \mu (m_\xi))^\blanco
&\;\;=\;\;  \mu \partial (m_{\xi\blanco}) +\big( \mu  \partial (m_\xi) \big)^\blanco \\
\d \mu (m_{\xi\negro}) -( \d \mu (m_\xi) )^\negro 
&\;\; =\;\; \mu \partial  (m_{\xi\negro}) +\big( \mu  \partial (m_\xi) \big)^\negro 
\end{eqnarray*}
We prove these in the following Lemma.
\end{proof}

\begin{lema}\label{mupartial}
If $\xi \in \{\negro, \blanco\}^{n-1}$ 
then
\begin{eqnarray*}
\d \big(\mu(m_{\xi\blanco})\big) -  {\big(\d( \mu(m_{\xi})\big)}^\blanco  & = &(-1)^{n} \Big( \mu(m_\blanco) \circ_1 \mu(m_{\xi}) -  \mu(m_{\xi}) \circ_{n} \mu(m_\blanco) \Big) \\
\d \big(\mu(m_{\xi\negro})\big) -  {\big(\d( \mu(m_{\xi})\big)}^\negro  & = & (-1)^{n} \Big(\mu(m_\negro) \circ_1 \mu(m_{\xi}) -  \mu(m_{\xi}) \circ_{n} \mu(m_\negro) \Big)
\\
\mu\partial (m_{\xi \blanco}) - (\mu\partial m_\xi)^\blanco  & = & (-1)^{n} \Big( \mu(m_\blanco) \circ_1 \mu(m_{\xi}) -  \mu(m_{\xi}) \circ_{n} \mu(m_\blanco) \Big) \\
\mu\partial (m_{\xi \negro}) - (\mu\partial m_\xi)^\negro  & = & (-1)^{n} \Big(\mu(m_\negro) \circ_1 \mu(m_{\xi}) -  \mu(m_{\xi}) \circ_{n} \mu(m_\negro) \Big)
\end{eqnarray*}
\end{lema}
\begin{proof} 
The first two equations follow by applying Proposition~\ref{bncomp} to each cactus $u$ which appears as a term of $\mu(m_\xi)$. 

For $\xi'\in\{\blanco,\negro\}^{p-1}$,  $\xi''\in\{\blanco,\negro\}^{q-1}$, $p,q\geq2$, we observe that
$$
(\mu (m_{\xi'}\circ_i m_{\xi''})
)^\blanco=\begin{cases}\mu \bigl((-1)^qm_{\xi'\blanco}\circ_i m_{\xi''}\bigr)&(i<p)
\\ \mu \bigl((-1)^qm_{\xi'\blanco}\circ_p m_{\xi''}+m_{\xi'}\circ_p m_{\xi''\blanco})\bigr)
&(i=p)\end{cases}
$$
by applying Proposition~\ref{bncomp2} to each term $u'\circ_i u''$ of 
$\mu m_{\xi'}\circ_i \mu m_{\xi''}$.
Hence
$$
(\mu\partial m_\xi)^\blanco=\mu\biggl(\sum (-1)^{q(p-i+1)+i-1}
m_{\xi'\blanco}\circ_i m_{\xi''}\;+\;(-1)^{p-1}m_{\xi'}\circ_pm_{\xi''\blanco}\biggl)
$$
where the sum is over all decompositions $\xi=\xi'\circ_i\xi''$, $1\leq i\leq p$ as in~\eqref{partialxi}. 
On the other hand, the possible decompositions of $\xi \blanco$ are ${\blanco} \circ_1 \xi$, $\xi \circ_n {\blanco}$, ${\xi'\blanco} \circ_i \xi''$ and 
$\xi' \circ_p {\xi''\blanco}$ with $\xi', \xi''$ a decomposition of $\xi$ as before, and 
in the formula for $\partial (m_{\xi\blanco})$ these four types of decomposition appear with the following signs:
$$
(-1)^n
m_{\blanco}\circ_1m_\xi,
\;
(-1)^{n-1}
m_\xi\circ_nm_{\blanco},
\;
(-1)^{q(p-i+1)+i-1}
m_{\xi'\blanco}\circ_im_{\xi''},
\;
(-1)^{p-1}
m_\xi'\circ_pm_{\xi''\blanco}.$$
We therefore obtain the third equation
$$\mu\partial m_{\xi \blanco} - (\mu\partial m_\xi)^\blanco  \;\; = \;\; (-1)^n \;( \mu m_\blanco \circ_1 \mu m_\xi -  \mu m_\xi \circ_{n} \mu m_\blanco ) $$
as required. The computations for the fourth equation are identical.
\end{proof} 

\bigskip

We end with a calculation of $\mu(m_\xi)$ in arities 3, 4 and 5. We omit all commas from the notation.
\bigskip

\begin{tabular}{lcl}
$m_\xi
$ &  maps to & a linear combination of $u\in\C'_n$\\
\hline
$m_{\negro \negro\dots}$   & $\mapsto$  &  $0$\\
$m_{\blanco \blanco\dots }$ & $\mapsto$  &  $0$\\
\hline
$m_{\negro \blanco}$  & $\mapsto$  &  $\phantom{-} (1312)$ \\
$m_{\blanco \negro }$ & $\mapsto$  &  $-(2131)$ \\
\hline
$m_{\negro \blanco \negro}$  & $\mapsto$  &  $- (131412) - (131242) $ \\
$m_{\negro \blanco \blanco}$ & $\mapsto$  &  $- (141312)$ \\
$m_{\blanco \negro \negro}$  & $\mapsto$  &  $+ (213141)$ \\
$m_{\blanco \negro \blanco}$ & $\mapsto$  &  $+ (242131)  + (214131)$ \\
\hline
$m_{\negro \blanco \negro \negro}$  & $\mapsto$  &  $+ (13141512) + (13141252) + (13124252)$  \\
$m_{\negro \blanco \negro \blanco}$ & $\mapsto$  &  $- (15131412) + (13531412) + (13151412) - (15131242) $ \\
 & & $ + (13531242) + (13151242) + (13125242)$ \\
$m_{\negro \blanco \blanco \negro}$ & $\mapsto$  & $+ (14131252) + (14131512) + (14135312) - (14151312) $ \\
$m_{\negro \blanco \blanco \blanco}$& $\mapsto$  & $- (15141312)$ \\
$m_{\blanco \negro \negro \negro}$  & $\mapsto$  & $  - (21314151) $ \\
$m_{\blanco \negro \negro \blanco}$ & $\mapsto$  & $ + (25213141) + (21513141) - (21353141) - (21315141) $ \\
$m_{\blanco \negro \blanco \negro}$ & $\mapsto$  & $  + (24252131) + (24215131) - (24213531) - (24213151)$ \\
& & $ + (21415131) - (21413531) - (21413151)$ \\
$m_{\blanco \negro \blanco \blanco}$& $\mapsto$  & $ + (25242131)  + (25214131) + (21514131) $ \\
\end{tabular}

\section*{Acknowledgements.} 
The first and third authors were partially supported by Spanish Ministry of Science and Innovation grants
MTM2010-20692 
and
MTM2010-15831 
and by the Generalitat de Catalunya under the grants
SGR-1092-2009 
and
SGR-119-2009. 
They also thank the Isaac Newton Institute for Mathematical Sciences, where this paper was completed during the programme on Grothendieck--Teichm\"uller Groups, Deformation and Operads.
The second author would like to thank the       Universitat Polit\`ecnica de Catalunya  and the Universidad de Barcelona
for their hospitality.
\bibliographystyle{plain}
\bibliography{GLT}
\end{document}